\documentclass[reqno]{amsart}
\usepackage{amssymb, amsthm, amsmath, amsfonts}
\usepackage{graphics}
\usepackage{hyperref}
\usepackage[all]{xy}
\usepackage{enumerate}
\usepackage[mathscr]{eucal}

\theoremstyle{plain}
\newtheorem{theorem}{Theorem}[section]
\newtheorem{lemma}[theorem]{Lemma}
\newtheorem{proposition}[theorem]{Proposition}
\newtheorem{corollary}[theorem]{Corollary}
\numberwithin{equation}{section}
\theoremstyle{definition}

\newtheorem{definition}[theorem]{Definition}

\DeclareMathOperator{\Ob}{Ob}
\DeclareMathOperator{\Hom}{Hom}
\DeclareMathOperator{\Ker}{Ker}
\DeclareMathOperator{\op}{op}

\newcommand{\C}{{\mathscr{C}}}
\newcommand{\F}{{\mathbb{F}}}
\newcommand{\FI}{{\mathrm{FI}}}
\newcommand{\VI}{{\mathrm{VI}}}

\title{An application of Nakayama functor in representation stability theory}

\author{Wee Liang Gan}
\address{Department of Mathematics, University of California, Riverside, CA 92521, USA}
\email{wlgan@ucr.edu}

\author{Liping Li}
\address{Key Laboratory of High Performance Computing and Stochastic Information Processing (Ministry of Education), College of Mathematics and Computer Science, Hunan Normal University, Changsha, Hunan 410081, China}
\email{lipingli@hunnu.edu.cn}

\author{Changchang Xi}
\address{School of Mathematical Sciences, Capital Normal University, 100048 Beijing, China}
\email{xicc@cnu.edu.cn}

\thanks{L. Li is supported by the National Natural Science Foundation of China 11771135, the Construct Program of the Key Discipline in Hunan Province, and the Start-Up Funds of Hunan Normal University 830122-0037, while C.C.~Xi is partially supported by the National Natural Science Foundation of China 11331006.}


\begin{document}

\begin{abstract}
Using the Nakayama functor, we construct an equivalence from a Serre quotient category of a category of finitely generated modules to a category of finite-dimensional modules.  We then apply this result to the categories FI$_G$ and VI$_q$, and answer positively an open question of Nagpal on representation stability theory.
\end{abstract}

\maketitle

\section{Introduction}

In a recent paper \cite{nagpal}, Nagpal has proved that quite a few representation theoretic and homological properties of the category $\FI$, whose objects are finite sets and morphisms are injections (see \cite{cef}), hold for the category $\VI_q$, whose objects are finite-dimensional vector spaces over a finite field $\F_q$ and morphisms are linear injections. In that paper, he also asked a question of whether one can establish an equivalence
\begin{equation*}
\VI_q \mbox{-mod} / \VI_q \mbox{-fdmod} \xrightarrow{\quad\sim\quad} \VI_q \mbox{-fdmod},
\end{equation*}
where $\VI_q \mbox{-mod}$ is the category of finitely generated $\VI_q$-modules over a field of characteristic zero, and $\VI_q \mbox{-fdmod}$ is its full subcategory of finite-dimensional $\VI_q$-modules; see \cite[Question 1.11]{nagpal}.  An analogue of this equivalence for $\FI$-modules was proved by Sam and Snowden in  \cite[Theorem 3.2.1]{ss2}.

The purpose of the present paper is to prove a general result in an abstract setting from which this kind of equivalences can be deduced. In particular, applying our general result, the above equivalence can be obtained for both $\FI_G$ and $\VI_q$ with $G$ an arbitrary finite group (see Section \ref{question} for definition). Therefore, we give not only an affirmative answer to the above-mentioned open question in \cite[Question 1.11]{nagpal}, but also a new proof for the case of the category $\FI$ when taking $G$ to be trivial in $\FI_G$. Our approach only relies on several abstract homological properties of representations, and hence works for a wider class of categories including $\FI_G$ and $\VI_q$ as specific examples. Furthermore, via the Nakayama functor, the above equivalence for $\FI$ becomes transparent in our approach, compared with the one in \cite{ss2}.

We briefly describe the essential idea of our approach. Let $\C$ be a small EI-category (EI means that every endomorphism is an isomorphism) satisfying certain finiteness conditions. One can define the \emph{Nakayama functor} $\nu$ and \emph{inverse Nakayama functor} $\nu^{-1}$
\begin{equation*}
\xymatrix{ \C\mbox{\rm{-mod}} \ar[rr]<.5ex>^{\nu} && \C\mbox{\rm{-fdmod}} \ar[ll]<0.5ex>^{\nu^{-1}}  }
\end{equation*}
between the category $\C\mbox{\rm{-mod}}$ of finitely generated $\C$-modules and the category $\C\mbox{\rm{-fdmod}}$ of finite-dimensional $\C$-modules. Note that $\nu$ and $\nu^{-1}$ form a pair of adjoint functors, and furthermore, they give rise to an equivalence between the category of finitely generated projective $\C$-modules and the category of finite-dimensional injective modules. Under the assumption that $\C$ is \emph{locally self-injective} (that is, every finitely generated projective $\C$-module is also injective), $\nu$ is an exact functor, and $\nu \circ \nu^{-1} $ is isomorphic to the identity functor on $\C\mbox{\rm{-fdmod}}$. Therefore, by a classical result of Gabriel (\cite[Proposition III.2.5]{gabriel}), the kernel of $\nu$ is a localizing subcategory of $\C\mbox{\rm{-mod}}$, and one obtains the following commutative diagram in which $\bar{\nu}$ and $\bar{\nu}^{-1}$ are quasi-inverse to each other, $\rm{loc}$ is the localization functor, and $\rm{sec}$ is a section functor:
\begin{equation*}
\xymatrix{
\Ker(\nu) \ar[d]^{\rm{inc}}\\
\C\mbox{\rm{-mod}} \ar[rr]<.5ex>^{\nu} \ar[dd]<.5ex>^{\rm{loc}} & & \C\mbox{\rm{-fdmod}} \ar[ll]<.5ex>^{\nu^{-1}} \ar@{-->}[lldd]<.5ex>^{\bar{\nu}^{-1}}\\
 \\
\C\mbox{\rm{-mod}} / \Ker(\nu) \ar[uu]<.5ex>^{\rm{sec}} \ar@{-->}[rruu]<.5ex>^{\bar{\nu}}
}
\end{equation*}

We then consider $\Ker(\nu)$. If every morphism in $\C$ is a monomorphism, then $\C\mbox{\rm{-fdmod}} \subseteq \Ker(\nu)$. We also give two equivalent characterizations such that $\Ker(\nu) \subseteq \C\mbox{\rm{-fdmod}}$. When the category $\C$ satisfies these conditions (for instances, $\C$ is a skeleton of $\FI_G$ or $\VI_q$), the above commutative diagram becomes
\begin{equation*}
\xymatrix{
\C\mbox{\rm{-fdmod}} \ar[d]^{\rm{inc}}\\
\C\mbox{\rm{-mod}} \ar[rr]<.5ex>^{\nu} \ar[dd]<.5ex>^{\rm{loc}} & & \C\mbox{\rm{-fdmod}} \ar[ll]<.5ex>^{\nu^{-1}} \ar@{-->}[lldd]<.5ex>^{\bar{\nu}^{-1}}\\
 \\
\C\mbox{\rm{-mod}} / \C\mbox{\rm{-fdmod}} \ar[uu]<.5ex>^{\rm{sec}} \ar@{-->}[rruu]<.5ex>^{\bar{\nu}}
}
\end{equation*}
Thus we get what we want.

The paper is organized as follows. In Section \ref{preliminaries}, we collect basic results on Nakayama functor. In Section \ref{main result} and  Section \ref{question}, we prove our main result and consider its application to both $\FI_G$ and $\VI_q$ in representation stability theory, respectively.

\section{Preliminaries} \label{preliminaries}

The Nakayama functor is well known in the representation theory of finite-dimensional algebras (see, for example, \cite{ars}, \cite{sy}). Most of the proofs in this section are standard, so we leave the details to the reader.

\subsection{Notations }
An EI-category is a small category in which every endomorphism is an isomorphism. Let $\C$ be a skeletal EI-category, and $I$ be the set of objects of $\C$. For any $i,j\in I$, we write $\C(i,j)$ for the set of morphisms in $\C$ from $i$ to $j$. Recall that there is a partial order on $I$ defined by $i\leqslant j$ if $\C(i,j)$ is nonempty.

We fix a field $\Bbbk$. A left (respectively, right) $\C$-module is a covariant (respectively, contravariant) functor $W$ from $\C$ to the category of $\Bbbk$-vector spaces.

For any $i,j\in I$, denote by $\Bbbk \C(i,j)$ the vector space with basis $\C(i,j)$. Denote by $A=\bigoplus_{i,j\in I} \Bbbk\C(i,j)$ the category algebra of $\C$. The associative algebra $A$ is non-unital if $I$ is an infinite set. Denote by $e_i\in \C(i,i)$ the identity endomorphism of $i$. We say that a left (respectively, right) $A$-module $V$ is \emph{graded} if $V = \bigoplus_{i\in I} e_i V$ (respectively, $V=\bigoplus_{i\in I} V e_i$). If $W$ is a left (respectively, right) $\C$-module, then $\bigoplus_{i\in I} W(i)$ has a natural structure of a graded left (respectively, right) $A$-module. Conversely, any graded left (respectively, right) $A$-module naturally defines a left (respectively, right) $\C$-module. Thus, we shall not distinguish left (respectively, right) $\C$-modules from graded left (respectively, right) $A$-modules.

A (left or right) $\C$-module is finitely generated (respectively, finite-dimensional) if it is
finitely generated (respectively, finite-dimensional) as a graded $A$-module. We write $\C\mbox{-mod}$ (respectively, $\C^{\op}\mbox{-mod}$) for the category of finitely generated left (respectively, right) $\C$-modules; and $\C\mbox{-fdmod}$ (respectively, $\C^{\op}\mbox{-fdmod}$) for the full subcategory of $\C\mbox{-mod}$ (respectively, $\C^{\op}\mbox{-mod}$) whose objects are the finite-dimensional left (respectively, right) $\C$-modules. On the categories of finite-dimensional modules, there is the \emph{standard duality functor} $D:=\Hom_\Bbbk ( - , \Bbbk)$:
\begin{equation*}
\xymatrix{ \C\mbox{-fdmod} \ar[rr]<.5ex>^{D} && \C^{\op}\mbox{-fdmod} \ar[ll]<.5ex>^{D} }
\end{equation*}

\subsection{Finiteness conditions and Nakayama functor}

We first recall from \cite[Section 2]{ss1} that the category $\C$ is said to be
\emph{inwards finite} if, for each $j\in I$, there are only finitely many $i\in I$ such that $\C(i,j)$ is nonempty; and \emph{ hom-finite } if $\C(i,j)$ is a finite set for every $i, j\in I$.
Further, the category $\C$ is called \emph{locally noetherian} if every $\C$-submodule of each finitely generated left $\C$-module is also finitely generated.

{\bf From now on, we always assume that $\C$ is an inwards finite, hom-finite, and locally noetherian EI-category.}

Since $\C$ is inwards finite and hom-finite, the right projective $\C$-module $e_j A$ is finite-dimensional for each $j\in I$. Hence every finitely generated right $\C$-module is finite-dimensional.

Next, we introduce the Nakayama functor on $\C$-modules.

The category algebra $A$ of $\C$ is an $A$-bimodule which is graded as both a left  $A$-module and a right $A$-module. If $V$ (respectively, $W$) is a left (respectively, right) $\C$-module, then $\Hom_{\C} (V, A)$ (respectively, $\Hom_{\C^{\op}} (W, A)$) is a right (respectively, left) $A$-module. If, moreover, $V$ (respectively, $W$) is \emph{finitely generated}, then $\Hom_{\C} (V, A)$ (respectively, $\Hom_{\C^{\op}} (W, A)$) is graded, that is,
\begin{equation} \label{hom to direct sum}
\Hom_{\C} (V, A) \cong \bigoplus_{i\in I} \Hom_{\C} (V, A e_i)\quad
\Big(\mbox{respectively, }\Hom_{\C^{\op}} (W, A) \cong \bigoplus_{j\in I} \Hom_{\C^{\op}} (W, e_j A) \Big);
\end{equation}
see \cite[VIII.1.15]{js}. (The proof of the claim below (1) in \cite[VIII.1.15]{js} remains valid for our algebra $A$ even though $A$ might be non-unital.)

Without any reference, we shall use the following well-known fact: for any idempotent $e\in A$, there hold:
\begin{equation} \label{idempotent}
\Hom_{\C} (Ae, A)_A \cong eA_A \qquad\mbox{ and }\qquad {}_A\Hom_{\C^{\op}} (eA, A) \cong {}_AAe.
\end{equation}

\begin{lemma} \label{fd}
If $V$ is a finitely generated left $\C$-module, then $\Hom_\C (V,A)$ is a finite-dimensional right $\C$-module.
\end{lemma}
\begin{proof}
Suppose that $V$ is generated by $v_1, \ldots, v_s$ where $v_1\in V(j_1), \dots, v_s\in V(j_s)$. Then $\Hom_\C(V, Ae_i)=0$ if $\C(i,j_1),\ldots, \C(i,j_s)$ are all empty sets. Since $\C$ is inwards finite, there are only finitely many $i\in I$ such that $\Hom_\C(V, Ae_i)\ne 0$. Since $\C$ is hom-finite, each $\Hom_\C(V, Ae_i)$ is finite-dimensional.
\end{proof}

\begin{lemma} \label{fg}
If $W$ is a finite-dimensional right $\C$-module, then $\Hom_{\C^{\op}} (W, A)$ is a finitely generated left $\C$-module.
\end{lemma}
\begin{proof}
There exists a surjective homomorphism $e_{j_1} A \oplus \cdots \oplus e_{j_s} A \to W$ for some $j_1, \ldots, j_s \in I$. Applying the functor $\Hom_{\C^{\op}} ( - , A)$ and using \eqref{idempotent}, we obtain an injective homomorphism
\begin{equation*}
\Hom_{\C^{\op}} (W, A) \to Ae_{j_1} \oplus \cdots \oplus Ae_{j_s}.
\end{equation*}
Since $\C$ is locally noetherian, it follows that $\Hom_{\C^{\op}} (W, A)$ is finitely generated.
\end{proof}

By Lemmas \ref{fd} and \ref{fg}, we have a pair of contravariant functors
\begin{equation*}
\xymatrix{ \C\mbox{-mod} \ar[rrr]<.5ex>^{\Hom_{\C} (-, A)} &&& \C^{\op}\mbox{-fdmod} \ar[lll]<.5ex>^{\Hom_{\C^{\op}} (-, A)}}
\end{equation*}

\begin{definition}
The \emph{Nakayama functor} $\nu$ of $\C$ (or $A$) is defined to be the following composition:
\begin{equation*}
D \circ \Hom_{\C} ( - , A) : \C\mbox{-mod} \longrightarrow \C\mbox{-fdmod}.
\end{equation*}
The \emph{inverse Nakayama functor} $\nu^{-1}$ is defined to be the following composition:
\begin{equation*}
\Hom_{\C^{\op}} ( - , A) \circ D : \C\mbox{-fdmod} \longrightarrow \C\mbox{-mod}.
\end{equation*}
\end{definition}

Let us remark that the functor $\nu$ is a right exact covariant functor, while the functor $\nu^{-1}$ is a left exact covariant functor. But we should warn the reader that the functor $\nu^{-1}$ is, in general, neither the inverse nor a quasi-inverse of $\nu$.

\begin{lemma} \label{adjoint}
The pair $(\nu, \nu^{-1})$ is an adjoint pair of functors:
\begin{equation*}
\xymatrix{ \C\mbox{\rm{-mod}} \ar[rr]<.5ex>^{\nu} && \C\mbox{\rm{-fdmod}} \ar[ll]<.5ex>^{\nu^{-1}} . }
\end{equation*}
\end{lemma}
\begin{proof}
Let $V\in \Ob(\C\mbox{-mod})$ and $U\in \Ob( \C\mbox{-fdmod} )$. Since $V$ is a left $A$-module and $DU$ is a right $A$-module, the tensor product $V\otimes_{\Bbbk} DU$ is an $A$-bimodule. One has the following canonical isomorphisms:
\begin{multline*}
\Hom_\C (D \Hom_\C (V,A) , U ) = \Hom_{\C^{\op}} (DU, \Hom_\C (V,A)) = \Hom_{A\text{-bimod}} (V\otimes_{\Bbbk} DU, A) \\ = \Hom_\C (V,  \Hom_{\C^{\op}} (DU , A) );
\end{multline*}
see \cite[Exercise XI.6.6]{grillet}.
\end{proof}

\subsection{Projectives and finite-dimensional injectives}
Denote by $\C\mbox{-proj}$ the full subcategory of $\C\mbox{-mod}$ whose objects are the finitely generated projective left $\C$-modules. Denote by $\C\mbox{-fdinj}$ the full subcategory of $\C\mbox{-fdmod}$ whose objects are the finite-dimensional injective left $\C$-modules.

\begin{lemma} \label{projectives}
(1) Every finitely generated projective left $\C$-module is a finite direct sum of indecomposable projective left $\C$-modules.

(2) Let $V$ be a finitely generated left $\C$-module. Then $V$ is an indecomposable projective left $\C$-module if and only if $V$ is isomorphic to $Ae$ for some primitive idempotent $e\in e_iAe_i$ with $i\in I$.
\end{lemma}
\begin{proof} The two statement follows from \cite[Theorem I.11.18]{dieck} and \cite[Proposition I.8.2]{sy}, respectively.
\end{proof}

\begin{definition}
(1) A full subcategory $\C'$ of $\C$ is said to be \emph{right-closed} if, for any $i,j\in \Ob(\C)$, we have $i\in \Ob(\C')$ whenever $i\leqslant j$ for some $j\in \Ob(\C')$.

(2) The \emph{support} of a (left or right) $\C$-module $V$ is the set of all $i\in I$ such that $V(i)$ is nonzero, where $V(i)$ is the image of $i$ under the functor $V$.
\end{definition}

By definition, we have the following trivial observation.
\begin{lemma} \label{subcategory}
Let $\C'$ be a right-closed subcategory of $\C$. If $V$ is an injective left $\C$-module whose support is contained in $\Ob(\C')$, then restricting $V$ to $\C'$ gives an injective left $\C'$-module. If $V'$ is an injective left $\C'$-module, then extending $V'$ to $\C$ by zero on $\Ob(\C)\setminus \Ob(\C')$ gives an injective left $\C$-module.
\end{lemma}

\begin{lemma} \label{fd injectives}
Let $U$ be a finite-dimensional left $\C$-module. Then $U$ is an indecomposable injective left $\C$-module if and only if $U$ is isomorphic to $D(eA)$ for some primitive idempotent $e\in e_iAe_i$ with $i\in I$.
\end{lemma}
\begin{proof}
This follows from Lemma \ref{subcategory} and \cite[Proposition I.8.19]{sy}.
\end{proof}

\begin{corollary}  \label{equiv}
The functor $\nu$ gives an equivalence of categories
\begin{equation*}
\C\mbox{{-\rm{proj}}} \xrightarrow{\quad\sim\quad} \C\mbox{-\rm{fdinj}}
\end{equation*}
with a quasi-inverse given by the functor $\nu^{-1}$.
\end{corollary}
\begin{proof}
This follows immediately from \eqref{idempotent}, Lemmas \ref{projectives} and \ref{fd injectives}.
\end{proof}

\begin{definition}
We say that $\C$ is \emph{locally self-injective} if $Ae_i$ is an injective left $\C$-module for every $i\in I$.
\end{definition}

Clearly, $\C$ is locally self-injective if and only if every finitely generated projective left $\C$-module is injective.

\begin{lemma} \label{exact}
Suppose that $\C$ is locally self-injective. Then the Nakayama functor $\nu: \C\mbox{-\rm{mod}} \longrightarrow \C\mbox{-\rm{fdmod}}$ is exact.
\end{lemma}
\begin{proof} 
Since $Ae_i$ is an injective left $\C$-module for each $i\in I$, the functor $\bigoplus_{i\in I} \Hom_\C( - , A e_i) $ is exact. It follows from \eqref{hom to direct sum} and the exactness of $D$ that $\nu: \C\mbox{-\rm{mod}} \longrightarrow \C\mbox{-\rm{fdmod}}$ is exact.
\end{proof}

\section{Main result} \label{main result}

\subsection{Injective resolutions of finite-dimensional modules}
The partial order on $I$ induces a partial order on the set of objects of any full subcategory of $\C$.

\begin{lemma} \label{projective resolution}
Suppose that the characteristic of $\Bbbk$ is zero. Let $\C'$ be a right-closed subcategory of $\C$ such that $\Ob(\C')$ is a finite set. Let $\C''$ be the full subcategory of $\C'$ on the objects which are not maximal in $\Ob(\C')$. Let $W$ be a finite-dimensional right $\C$-module whose support is contained in $\Ob(\C')$. Then:

(1) The subcategory $\C''$ of $\C$ is right-closed.

(2) There exists a short exact sequence
\begin{equation*}
0 \to W' \to P \to W \to 0
\end{equation*}
where $P$ is a finite direct sum of right $\C$-modules of the form $eA$ with $e^2=e\in e_i A e_i$ and $i\in \Ob(\C')$, and where $W'$ is a finite-dimensional right $\C$-module whose support is contained in $\Ob(\C'')$.
\end{lemma}
\begin{proof}
(1) Suppose $i\in \Ob(\C)$ and $i\leqslant j$ for some $j\in \Ob(\C'')$. Then $j\in \Ob(C')$, and $j$ is not a maximal object in $\Ob(\C')$. Hence $i\in \Ob(C')$, and $i$ is not a maximal object in $\Ob(\C')$.

(2) Let
\begin{equation*}
P = \bigoplus_{i\in \Ob(\C')} We_i \otimes_{e_i A e_i } e_i A.
\end{equation*}
The algebra $e_i A e_i$ is the group algebra of the finite group $\mathrm{Aut}_\C(i)$ for each $i\in I$. Since $\Bbbk$ has characteristic zero, the algebra $e_iAe_i$ is semisimple. It follows that $We_i$ is a finite direct sum of irreducible right $e_iAe_i$-modules, each of which is isomorphic to
$eAe_i$ for some primitive idempotent $e \in e_iAe_i$. One has a canonical isomorphism of right $\C$-modules: $eAe_i \otimes_{e_i A e_i} e_i A \cong eA$.
We see that $P$ is of the required form.

The multiplication map $\rho : P \to W$ is a homomorphism of right $A$-modules. Let $W'$ be the kernel of $\rho$. Since $P$ is finite-dimensional and has support contained in $\Ob(\C')$, the same is true for $W'$.

For each $i\in \Ob(\C')$, since $\rho$ maps $We_i \otimes_{e_i A e_i} e_i A e_i$ bijectively to $W e_i$, we see that $\rho$ is surjective. Moreover, if $i$ is maximal in $\Ob(\C')$, then $We_j \otimes_{e_j A e_j} e_j A e_i = 0$ if $j\in \Ob(\C')$ and $j\neq i$. Therefore if $i$ is maximal in $\Ob(\C')$, then $\rho$ maps $Pe_i$ bijectively to $We_i$. It follows that $W'$ has support contained in $\Ob(\C'')$.
\end{proof}

\begin{lemma} \label{injective resolution}
Suppose that the characteristic of $\Bbbk$ is zero. Let $U$ be a finite-dimensional left $\C$-module. Then there exists an exact sequence
\begin{equation*}
0 \to U \to I_0 \to I_1 \to \cdots \to I_n \to 0
\end{equation*}
of left $\C$-modules such that $I_0, I_1, \ldots, I_n$ are finite-dimensional injective left $\C$-modules.
\end{lemma}
\begin{proof}
Let $W$ be the right $\C$-module $DU$. Let $\C_0$ be the full subcategory of $\C$ on the objects $i$ such that $i\leqslant j$ for some $j$ in the support of $W$. It is clear that $\C_0$ is right-closed.  Since the support of $W$ is a finite set and $\C$ is inwards finite, the set $\Ob(\C_0)$ is finite. Let $\C_1$ be the full subcategory of $\C_0$ on the objects which are not maximal in $\C_0$. Then $\C_1$ is also a right-closed subcategory of $\C$ with $\Ob(\C_1)$ finite. By Lemma \ref{projective resolution}, there is a short exact sequence
\begin{equation*}
0 \to W_1 \to P_0 \to W \to 0
\end{equation*}
of right $\C$-modules such that:
\begin{itemize}
\item $P_0$ is a finite direct sum of right $\C$-modules of the form $eA$ for some idempotent $e\in e_i A e_i$ with $i\in \Ob(C_0)$;

\item $W_1$ is finite dimensional and its support is contained in $\Ob(\C_1)$.
\end{itemize}
Let $\C_2$ be the full subcategory of $\C_1$ on the objects which are not maximal in $\C_1$. We now apply Lemma \ref{projective resolution} again to obtain a short exact sequence
\begin{equation*}
0 \to W_2 \to P_1 \to W_1 \to 0
\end{equation*}
of right $\C$-modules such that:
\begin{itemize}
\item $P_1$ is a finite direct sum of right $\C$-modules of the form $eA$ for some idempotent $e\in e_i A e_i$ with $i\in \Ob(C_1)$;

\item $W_2$ is finite dimensional and its support is contained in $\Ob(\C_2)$.
\end{itemize}
Recursively, we obtain the subcategories $\C_0, \C_1, \C_2, \ldots$ and a projective resolution
\begin{equation} \label{resolution}
\cdots \to P_n \to \cdots \to P_1 \to P_0 \to W \to 0
\end{equation}
where each $P_s$ is a finite direct sum of right $\C$-modules of the form $eA$ for some idempotent $e\in e_i A e_i$ with $i\in \Ob(\C_s)$. Since $|\Ob(\C_0)|>|\Ob(\C_1)|>|\Ob(\C_2)| > \cdots$, the projective resolution \eqref{resolution} is of finite length. By Lemma \ref{fd injectives}, we see that applying the functor $D$ to \eqref{resolution} gives the required exact sequence.
\end{proof}

\subsection{An equivalence of categories}
For the adjoint pair $(\nu, \nu^{-1})$ in Lemma \ref{adjoint}, we have the following result.

\begin{proposition} \label{section functor}
Suppose that the characteristic of $\Bbbk$ is zero and $\C$ is locally self-injective. Then the counit $\nu \circ \nu^{-1} \to \mathrm{id}$ is an isomorphism.
\end{proposition}
\begin{proof}
We need to prove that the homomorphism $\nu(\nu^{-1}(U)) \to U$ is an isomorphism for each object $U$ of $\C\mbox{-fdmod}$. By Lemma \ref{injective resolution}, there is a finite injective resolution of $U$ by finite-dimensional injective left $\C$-modules; we shall use induction on the minimal length $n$ among all such resolutions. For $n=0$, the proposition follows from Corollary \ref{equiv}.

Let $0 \to U \to I_0 \to I_1 \to \cdots \to I_n \to 0$ be a resolution of $U$ by finite-dimensional injective left $\C$-modules and $n$ be minimal length with this property. Let $U'$ be the cokernel of $U \to I_0$. One gets a short exact sequence
\begin{equation*}
0 \to U \to I_0 \to U' \to 0.
\end{equation*}

Since $\nu^{-1}$ is left exact and $\nu$ is exact by Lemma \ref{exact}, the following diagram commutes and has exact rows:
\begin{equation*}
\xymatrix{ 0 \ar[r] & \nu(\nu^{-1}(U)) \ar[r] \ar[d] & \nu(\nu^{-1}(I_0)) \ar[r] \ar[d] &\nu(\nu^{-1}(U')) \ar[d] \\
0 \ar[r] & U \ar[r] & I_0 \ar[r] & U'  . }
\end{equation*}
Observe that the middle vertical map is an isomorphism by Corollary \ref{equiv} and the right vertical map is an isomorphism by induction hypothesis. Therefore the left vertical map is also an isomorphism.
\end{proof}

The kernel $\Ker(\nu)$ of the functor $\nu: \C\mbox{-mod} \to \C\mbox{-fdmod}$ is the full subcategory of $\C\mbox{-mod}$ on the objects $V$ such that $\nu(V)=0$. When $\nu$ is an exact functor, $\Ker(\nu)$ is a Serre subcategory of $\C\mbox{-mod}$ (that is, closed under submodules, quotients and extensions) and its Serre quotient category is denoted by $\C\text{-mod}/\Ker(\nu)$.

To establish the main result of this paper, we recall the definition of section functors and a classical result of Gabriel.

\begin{definition}
We say an exact functor $F: \mathcal{A} \to \mathcal{B}$ between two abelian categories admits a \emph{section functor} if $F$ has a right adjoint $S: \mathcal{B} \to \mathcal{A}$ such that the counit $F \circ S \to \rm{id}_{\mathcal{B}}$ is an isomorphism.
\end{definition}

\begin{lemma}\cite[Proposition III.2.5]{gabriel} \label{gabriel}
Let $F: \mathcal{A} \to \mathcal{B}$ be an exact functor between abelian categories which admits a section functor. Then the kernel $\Ker(F)$ is a localizing subcategory of $\mathcal{A}$, and $F$ induces an equivalence $\bar{F}: \mathcal{A} / \Ker(F) \to \mathcal{B}$.
\end{lemma}

Now we can prove the following main theorem.

\begin{theorem} \label{main theorem}
Suppose that the characteristic of $\Bbbk$ is zero, and suppose that $\C$ is locally self-injective. Then the functor $\nu$ induces an equivalence of categories
\begin{equation*}
\overline{\nu} : \C\mbox{-\rm{mod}}/\Ker(\nu)  \xrightarrow{\quad\sim\quad} \C\mbox{-\rm{fdmod}}.
\end{equation*}
\end{theorem}

\begin{proof}
Lemma \ref{exact} and Proposition \ref{section functor} tell us that the Nakayama functor $\nu$ is exact and admits a section functor $\nu^{-1}$. The conclusion now follows from Lemma \ref{gabriel}.
\end{proof}

\subsection{Observations on \texorpdfstring{$\Ker(\nu)$}{}}

In this subsection we compare the subcategory $\C\mbox{-\rm{fdmod}}$ with the subcategory $\Ker(\nu)$, and prove that they coincide under certain conditions. In this case, Theorem \ref{main theorem} says that the Serre quotient category $\C\mbox{-\rm{mod}} / \C\mbox{-\rm{fdmod}}$ is equivalent to $\C\mbox{-\rm{fdmod}}$. Therefore the Serre quotient category has enough injective objects, and every finitely generated object in it has finite length and finite injective dimension. Moreover, $\C\mbox{-\rm{mod}}$ can be regarded as an extension of $\C\mbox{-\rm{fdmod}}$ by itself.

Firstly, we give a sufficient condition such that $\C\mbox{-\rm{fdmod}} \subseteq \Ker(\nu)$, which should be easy to check in practice.

\begin{definition}
We say that a left $\C$-module $V$ is \emph{torsion-free} if, for every morphism $f$ in $\C$, say $f\in \C(i,j)$, the induced map $f_* : V(i) \to V(j)$ is injective.
\end{definition}

\begin{lemma} \label{monomorphism}
The following statements are equivalent:

(1) Every morphism in $\C$ is a monomorphism, that is, if one has $f\in \C(j,k)$ and $g,h\in \C(i,j)$ such that $fg=fh$, then $g=h$.

(2) Every finitely generated projective left $\C$-module is torsion-free.
\end{lemma}
\begin{proof}
(1)$\Rightarrow$(2) It suffices to prove that $Ae_i$ is torsion-free for each $i\in I$. Suppose $f\in \C(j,k)$. Since $f$ is a monomorphism, the map $f_*: e_j A e_i \to e_k A e_i$ sends the basis $\C(i,j)$ of $e_jAe_i$ bijectively onto a subset of the basis $\C(i,k)$ of $e_kAe_i$.

(2)$\Rightarrow$(1) Suppose one has $f\in \C(j,k)$ and $g,h\in \C(i,j)$ such that $fg=fh$. Since $Ae_i$ is torsion-free, the map $f_*: e_j A e_i \to e_k A e_i$ is injective. But $f_* (g-h)=0$, so $g=h$.
\end{proof}

\begin{corollary} \label{kernel}
Suppose that the partially ordered set $I$ has no maximal element. If every morphism in $\C$ is a monomorphism, then
$\C\mbox{-\rm{fdmod}} \subseteq \Ker(\nu).$
\end{corollary}

\begin{proof}
Let $V$ be a finite-dimensional left $\C$-module. We need to show that $\Hom_\C (V, Ae_i)=0$ for every $i\in I$. By Lemma \ref{monomorphism}, the $\C$-module $Ae_i$ is torsion-free, so it does not contain any nonzero finite-dimensional $\C$-submodules. Hence, the image of any homomorphism from $V$ to $Ae_i$ is zero.
\end{proof}

The following proposition gives two equivalent characterizations such that the reverse inclusion $\Ker(\nu) \subseteq \C\mbox{-\rm{fdmod}}$ holds. Surprisingly, the answer to this question is closely related to classification of injective modules in $\C\mbox{-\rm{mod}}$.

\begin{proposition} \label{inverse inclusion}
Suppose that the characteristic of $\Bbbk$ is zero and $\C$ is locally self-injective. Then the following statements are equivalent:
\begin{enumerate}
\item $\Ker(\nu) \subseteq \C\mbox{-\rm{fdmod}}$.
\item If $V$ is a finitely generated, infinite-dimensional left $\C$-module, then there exists $i \in I$ such that $\Hom_{\C} (V, Ae_i) \neq 0$.
\item The category $\C\mbox{-\rm{mod}}$ has enough injectives, and every finitely generated injective left $\C$-module is isomorphic to a direct sum of a finite-dimensional injective left $\C$-module and a finitely generated projective left $\C$-module.
\end{enumerate}
\end{proposition}

\begin{proof}
The equivalence between (1) and (2) is clear. Note that if $\C$ has only finitely many objects, then $\C\mbox{-\rm{fdmod}} = \C\mbox{-\rm{mod}}$, and there does not exist finitely generated, infinite-dimensional left $\C$-modules. Therefore, in this case both (1) and (2) hold trivially.

$(3) \Rightarrow (2)$ Let $V$ be a finitely generated, infinite-dimensional left $\C$-module. By the assumption, there exists an injection $V \to P \oplus T$, where $P$ is a finitely generated projective $\C$-module, and $T$ is a finite-dimensional injective $\C$-module. Note that $P$ is nonzero since $V$ is infinite-dimensional. Furthermore, the composition map $V \to P \oplus T \to P$ cannot be 0 where the second one is the projection because of the same reason. This observation implies (2).

$(2) \Rightarrow (3)$ Suppose that (2) (and hence (1)) holds. Let us first prove the following statement:

\begin{itemize}
\item[$(\boldsymbol{\star})$] If $F$ is a finitely generated left $\C$-module which has no nonzero finite-dimensional $\C$-submodules, then there is an injective homomorphism $F \to P$, where $P$ is a finitely generated projective left $\C$-module.
\end{itemize}
As in \cite[Proposition 7.5]{gl2}, we use induction on the dimension of $\nu(F)$. If $\nu(F)=0$, then, by (1), $F$ is finite-dimensional and so $F=0$. Now, suppose $\nu(F)\ne 0$. By assumption, $F$ is infinite-dimensional. It then follows from (2) that there exists $i \in I$ and a nonzero homomorphism $f: F \to Ae_i$. Let $W$ be the image of $F$ under $f$. Then we get a short exact sequence of finitely generated left $\C$-modules
\begin{equation*}
0 \to U \to F \to W \to 0,
\end{equation*}
which induces another short exact sequence of finite-dimensional left $\C$-modules:
\begin{equation*}
0 \to \nu(U) \to \nu(F) \to \nu(W) \to 0.
\end{equation*}
Note that $\Hom_\C(W,Ae_i)\ne 0$ since there is an inclusion from $W$ into $Ae_i$. In particular, $\nu(W)\neq 0$. Thus the dimension of $\nu(U)$ is strictly less than that of $\nu(F)$. By induction hypothesis, $(\boldsymbol{\star})$ holds for $U$. Hence $(\boldsymbol{\star})$ holds for $F$.

Now let $V$ be any finitely generated left $\C$-module. Let $E$ be the maximal finite-dimensional $\C$-submodule of $V$. By Lemma \ref{injective resolution}, there is an injection $E\to T$ where $T$ is a finite-dimensional injective left $\C$-module. Let $F=V/E$. Then $F$ is a finitely generated left $\C$-module which has no nonzero finite-dimensional $\C$-submodules, so by $(\boldsymbol{\star})$, there is an injection $F\to P$ where $P$ is a finitely generated projective left $\C$-module. It follows that there is an injection $V \to T\oplus P$. Therefore the category $\C\mbox{-\rm{mod}}$ has enough injectives.

As in the proof of \cite[Lemma 2.4]{nagpal}, we suppose that the module $V$ is injective and $E'$ is a maximal essential extension of $E$ in $V$. Then $E'$ is injective by \cite[Proposition X.5.4]{grillet}. Moreover, $E'$ must be finite-dimensional, for otherwise it cannot be an essential extension of $E$. Hence $E=E'$, and so $E$ is injective. Therefore, $V$ is isomorphic to $E\oplus F$. Hence $F$ is also injective. By $(\boldsymbol{\star})$, it follows that $F$ is projective.
\end{proof}

\section{On an open question of Nagpal \label{question}}

We now discuss the application of Theorem \ref{main theorem} to the categories $\FI_G$ and $\VI_q$ studied in representation stability theory. In particular, we affirmatively answer the following open question of Nagpal \cite[Question 1.11]{nagpal}:

{\bf Question:} Can one establish an equivalence 
\begin{equation*}
\VI_q \mbox{-mod} / \VI_q \mbox{-fdmod} \xrightarrow{\quad\sim\quad} \VI_q \mbox{-fdmod} \; ?
\end{equation*}

Let $G$ be a finite group. The category $\FI_G$ was defined independently in \cite[Example 3.8]{gl1} and \cite{ss4}; let us first recall its definition. The objects of $\FI_G$ are the finite sets. The morphisms in $\FI_G$ from a finite set $X$ to a finite set $Y$ are the pairs $(f,c)$ where $f:X\to Y$ is an injection and $c: X\to G$ is an arbitrary map. If $(f,c)$ is a morphism from $X$ to $Y$, and $(f', c')$ is a morphism from $Y$ to $Z$, then their composite is defined to be the morphism $(f'', c'')$, where
\begin{equation*}
f''(x) = f'(f(x)), \qquad c''(x) = c'(f(x)) c(x), \qquad \mbox{ for each } x\in X.
\end{equation*}
When $G$ is the trivial group, the category $\FI_G$ is the category $\FI$ of finite sets and injections.

Now, we recall the definition of the category $\VI_q$. Let $\F_q$ be a finite field with $q$ elements. The objects of $\VI_q$ are finite-dimensional vector spaces over $\F_q$, and morphisms are the injective linear maps. The composite of morphisms in $\VI_q$ is just the usual composition of maps.
The category $\VI_q$ has been studied by many authors; see for example \cite{nagpal} and the references therein.

In the following lemma we collect some important results on $\FI_G$ and $\VI_q$, which will be used to establish an equivalence between the Serre quotient category and the category of finite-dimensional modules.

\begin{lemma} \label{basic results}
Suppose that $\Bbbk$ is of characteristic zero and $\C$ is a skeleton of $\FI_G$ or $\VI_q$. Then
\begin{enumerate}
\item $\C$ is locally noetherian over $\Bbbk$.
\item Every finitely generated projective left $\C$-module is injective.
\item Every morphism in $\C$ is a monomorphism.
\item If $V$ is a finitely generated, infinite-dimensional left $\C$-module, then there exists some $i \in I$ such that $\Hom_A (V, Ae_i) \neq 0$.
\end{enumerate}
\end{lemma}

\begin{proof}
(1) is established in \cite[Theorem 3.7, Example 3.8, Example 3.10]{gl1}. (2) is verified in \cite[Theorem 1.5]{gl2}.
(3) follows from the definitions of $\FI_G$ and $\VI_q$. (4) follows from both \cite[Lemma 7.2, Lemma 7.3]{gl2} for $\FI_G$ and from \cite[Theorem 4.34]{nagpal} for $\VI_q$.
\end{proof}

Now, we are ready to prove the following result for $\FI_G$ and $\VI_q$, which answers positively the above-mentioned question.

\begin{theorem} \label{fig and vi}
Suppose that the characteristic of $\Bbbk$ is zero and $\C$ is a skeleton of $\FI_G$ or $\VI_q$. Then the Nakayama functor $\nu$ induces an equivalence of categories
\begin{equation*}
 \C\mbox{-\rm{mod}}/\C\mbox{-\rm{fdmod}} \xrightarrow{\quad\sim\quad} \C\mbox{-\rm{fdmod}}.
\end{equation*}
\end{theorem}

\begin{proof}
It is clear that $\C$ is an EI-category which is inwards finite and hom-finite, and Lemma \ref{basic results}(1) tells us that $\C$ is a locally noetherian category. Clearly, the set of objects of $\C$ has no maximal element. By Lemma \ref{basic results}(2) and Theorem \ref{main theorem}, the Nakayama functor $\nu$ induces an equivalence of categories from $\C\mbox{-mod}/\Ker(\nu)$ to $\C\mbox{-fdmod}$. We claim $\Ker(\nu)=\C\mbox{-fdmod}$. Indeed, by Corollary \ref{kernel} and Lemma \ref{basic results}(3), one has $\C\mbox{-fdmod} \subseteq \Ker(\nu)$; and by Lemma \ref{basic results}(4) and Proposition \ref{inverse inclusion}, one gets $\Ker(\nu) \subseteq \C\mbox{-fdmod}$.
\end{proof}

\end{document}